\documentclass{article}

\usepackage{amsmath}
\usepackage{amssymb}
\usepackage{amsthm}
\usepackage{amsrefs}
\usepackage{bbm}
\usepackage{enumerate}
\usepackage{pifont}
\usepackage{setspace}
\usepackage[utf8]{inputenc}

\newcommand{\Z}{\mathbb{Z}}
\newcommand{\D}{\mathbb{D}}
 
\newcommand{\N}{\mathbb{N}} 
\newcommand{\R}{\mathbb{R}} 
\newcommand{\T}{\mathbb{T}} 
\newcommand{\C}{\mathbb{C}}
\newcommand{\1}{\mathbbm{1}}

\newcommand{\CC}{\mathcal{C}}
\newcommand{\DC}{\mathcal{D}}

\newcommand{\FC}{\mathcal{F}}

\newcommand{\HC}{\mathcal{H}}
\newcommand{\IC}{\mathcal{I}}

\newcommand{\LC}{\mathcal{L}}

\newcommand{\PC}{\mathcal{P}}

\newcommand{\SC}{\mathcal{S}}

\newcommand{\WC}{\mathcal{W}}

\newcommand{\rd}{\textnormal{rd}}
\newcommand{\rk}{\textnormal{rk}}

\newcommand{\dist}{\textnormal{dist}}

\newtheorem{theorem}{Theorem}[section]
\newtheorem{lemma}[theorem]{Lemma}

\newtheorem{proposition}[theorem]{Proposition}
\newtheorem{corollary}[theorem]{Corollary}

\newtheorem{remark}[theorem]{Remark}

\title{Two more counterexamples to the infinite dimensional Carleson embedding theorem}
\author{Eskil Rydhe}

\begin{document}

\maketitle

\begin{abstract}
	The existence of a counterexample to the infinite-dimensional Carleson embedding theorem has been established by Nazarov, Pisier, Treil, and Volberg. We provide an explicit construction of such an example. We also obtain a non-constructive example of particularly simple form; the density function of the measure (with respect to a certain weighted area measure) is the tensor-square of a Hilbert space-valued analytic function. This special structure of the measure has implications for Hankel-like operators appearing in control theory.
\end{abstract}

\section{Introduction}
Let $\HC$ denote a separable Hilbert space with norm $\|\cdot\|_\HC$ and inner product $\langle\cdot,\cdot\rangle_\HC$. We use $N\in[1,\infty]$ to denote the dimension of $\HC$, and $\LC_+(\HC)$ the set of positive (bounded) linear operators on $\HC$. We let $L^2(\T,\HC)$ denote the standard space of $2$-Lebesgue--Bochner integrable $\HC$-valued functions defined on the unit circle $\T$, and $H^2(\T,\HC)$ the subspace of analytic functions in $L^2(\T,\HC)$. 

Throughout this paper, we let $\mu$ be an $\LC_+(\HC)$-valued measure on the open unit disc $\D$. By $L^2(\D,\HC,d\mu)$ we denote the space of strongly measurable functions $f:\D\to\HC$ such that
\[
\int_{\D}\langle d\mu\, f,f\rangle_\HC <\infty.
\]

Given an arc $I\subset \T$, the corresponding Carleson square is the set $Q_I=\{w\in\D  ;\frac{w}{|w|}\in I,1-|I|<|w|<1\}$. Here $|I|$ denotes the normalized Lebesgue measure of $I$, i.e. $|\T|=1$. The Carleson intensity of $\mu$ is defined as
\[
\|\mu \|_\IC=\sup_{I\subset \T,\|e\|_\HC=1} \langle \mu(Q_I)e,e\rangle_\HC.
\]
Note that in order to obtain essentially the same quantity, it suffices to consider dyadic arcs.

We define the harmonic extension operator for integrable functions $f:\T\to\HC$ by
\[
\PC f(re^{2\pi i x})=\int_0^1f(e^{2\pi it}) P_r(x-t)\, dt,\quad w=re^{2\pi i x}\in \D,
\]
where
\[
P_r(t)=\frac{1-r^2}{1-2r\cos(2\pi t)+r^2},\quad t\in\R,
\]
is the usual Poisson kernel for $\D$. In the sequel, we shall typically write $f(w)$ as shorthand for $\PC f(w)$. 

The top tile of $Q_I$ is the set $T_I=\{w\in Q_I;|w|<1-\frac{|I|}{2}\}$. We define the dyadic extension operator by
\[
\PC^d f(w)=\sum_{I\in\DC} \1_{T_I}(w)\frac{1}{|I|}\int_I f\, dm,\quad w\in \D,
\]
where $\DC$ denotes the collection of dyadic arcs in $\T$, and $\1_{T_I}$ is the characteristic function of $T_I$.

Given $\mu$, we refer to $\PC: H^2(\HC)\to L^2(\D,\HC,d\mu)$ as a (harmonic) Carleson embedding. It is of interest to characterize the class of measures $\mu$ for which such embeddings are bounded. In the scalar-valued case, i.e. $N=1$, such measures are characterized by having finite Carleson intensity. Moreover, the corresponding norms are comparable. This characterization scales badly with the dimension of $\HC$. A first result in this direction is due to Nazarov, Treil, and Volberg \cite{Nazarov-Treil-Volberg1997:CounterExInfDimCarlesonEmbThm}:
\begin{proposition}\label{Proposition:NTV}
	There exists a universal constant $c>0$ with the following property: If $\HC$ is a Hilbert space of dimension $N<\infty$, then there exists an $\LC_+(\HC)$-valued measure $\mu$ on $\D$, such that
	\begin{equation*}
		\frac{\|\PC^d\|_{L^2(\T,\HC ) \to L^2(\D,\HC,d\mu)}}{\|\mu\|_{\IC}}\ge c (\log N)^{1/2}.
	\end{equation*}
\end{proposition}
Proposition \ref{Proposition:NTV} was proved using a rather sophisticated, yet explicit, construction. A corollary to this result is that if $\HC$ is infinite-dimensional, then there exists an $\LC_+(\HC)$-valued measure on $\D$, such that $\|\mu\|_\IC<\infty$, while $\PC^d:L^2(\T,\HC ) \to L^2(\D,\HC,d\mu)$ is unbounded. It has later been observed by Pott and Sadosky \cite{Pott-Sadosky2002:BMOBi-discOpBMO} that this result may be deduced from a geometric construction due to Carleson \cite{Carleson1974:CounterExMeasuresBddOnHpBi-disc}. 

A corresponding result for harmonic embeddings, along with a sharp estimate of the dimensional growth, was obtained by Nazarov, Pisier, Treil, and Volberg \cite{Nazarov-Pisier-Treil-Volberg2002:EstsVecCarlesonEmbThmVecParaprods}:
\begin{proposition}\label{Proposition:NPTV}
	There exists a universal constant $c>0$ with the following property: If $\HC$ is a Hilbert space of dimension $N<\infty$, then there exists an $\LC_+(\HC)$-valued measure on $\D$, such that
	\begin{equation*}
		\frac{\|\PC\|_{H^2(\T,\HC ) \to L^2(\D,\HC,d\mu)}}{\|\mu\|_{\IC}}\ge c \log N.
	\end{equation*}
\end{proposition}

The methods used in \cite{Nazarov-Pisier-Treil-Volberg2002:EstsVecCarlesonEmbThmVecParaprods} yields the existence of such measures, but not an explicit representation. The goal of this note is to adapt the explicit construction from \cite{Nazarov-Treil-Volberg1997:CounterExInfDimCarlesonEmbThm} to the setting of harmonic embeddings:	

\begin{theorem}\label{Theorem:Main}
	There exists a universal constant $c>0$ with the following property: If $\HC $ is a Hilbert space of dimension $N<\infty$, then there exists an $\LC_+(\HC)$-valued measure $\mu$ on $\D$, such that
	\begin{equation*}
		\frac{\|\PC\|_{H^2(\T,\HC ) \to L^2(\D,\HC,d\mu)}}{\|\mu\|_{\IC}}\ge c (\log N)^{1/2}.
	\end{equation*}
	The measure $\mu$ may be explicitly constructed in such a way that
	\[
	d\mu(w)=\phi(w)\otimes \phi(w)(1-|w|^2)\, dA(w),
	\]
	where $\phi:\D\to\HC$ is analytic.
\end{theorem}

Note that Theorem \ref{Theorem:Main} asserts a smaller estimate of dimensional growth than Proposition \ref{Proposition:NPTV}. It may still be that Theorem \ref{Theorem:Main} is sharp for measures with rank one-valued density function.

Apart from the explicit construction, the main novelty of this paper (compared to \cite{Nazarov-Pisier-Treil-Volberg2002:EstsVecCarlesonEmbThmVecParaprods}) is that the measure in Theorem \ref{Theorem:Main} has a very simple form. The original motivation for this paper was to study a certain class of Hankel-like operators appearing naturally in control theory, see \citelist{\cite{Jacob-Rydhe-Wynn2014:WeightWeissConjRKTGenHankOps}\cite{Jacob-Partington-Pott-Wynn2015:BAdmGContrSemGrps}}. In that setting, the particular form of the measure in Theorem \ref{Theorem:Main} is indeed critical.

We demonstrate two different ways of transferring Theorem \ref{Theorem:Main} to the case where $\HC$ is infinite-dimensional. The first one gives an explicit construction of the corresponding measure. We leave the proof as an exercise.
\begin{corollary}\label{Corollary:Explicit}
	Let $c>0$ be the universal constant, whose existence is guaranteed by Theorem \ref{Theorem:Main}. For each $N\in\N$, let $\HC_N$ denote a Hilbert space of dimension $N$, and let $\mu_N$ be a measure such that $\|\mu_N\|_{\IC}=1$, and
	\begin{equation*}
	\|\PC\|_{H^2(\T,\HC_N ) \to L^2(\D,\HC_N,d\mu_N)}\ge c (\log N)^{1/2}.
	\end{equation*}
	Let $\HC=\oplus_{N=1}^\infty \HC_N$, and $\mu=\oplus_{N=1}^\infty \mu_N$. Then $\|\mu\|_\IC=1$, while $\PC:H^2(\T,\HC) \to L^2(\D,\HC,d\mu)$ is unbounded.
\end{corollary}

A feature of Theorem \ref{Theorem:Main} which is lost in Corollary \ref{Corollary:Explicit} is the simple form of the measure. We can preserve this feature, at the cost of losing the explicit representation.
\begin{corollary}\label{Corollary:Existence}
	If $\HC $ is infinite dimensional, then there exists an $\LC_+(\HC)$-valued measure $\mu$ on $\D$, such that $\|\mu\|_\IC<\infty$, while $\PC:H^2(\T,\HC ) \to L^2(\D,\HC,d\mu)$ is unbounded. Furthermore, $\mu$ has the property that 
	\[
	d\mu(w)=\phi(w)\otimes \phi(w)(1-|w|^2)\, dA(w),
	\]
	where $\phi:\D\to\HC$ is analytic.
\end{corollary}

The paper is structured as follows: In Section \ref{Sec:Notation} we fix some further notation. In Section \ref{Sec:BMOA} we discuss how our results relate to a certain class of Hankel-type operators appearing naturally in control theory, and to some vector-valued generalizations of bounded mean oscillation. The discussion incidentally leads to a proof of Corollary \ref{Corollary:Existence}. In Section \ref{Sec:Proof} we present the proof of Theorem \ref{Theorem:Main}. Some parts of the paper are quite technical, and it is therefor written with the intention that the level of technicality should roughly be an increasing function of page number.

\section{Notation}\label{Sec:Notation}
We use the standard notation $\Z$, $\R$, and $\C$ for the respective rings of integers, real numbers, and complex numbers. By $\N$ we denote the set of strictly positive elements of $\Z$, while $\N\cup\{0\}$ is denoted by $\N_0$. We let $\D  =\{w\in\C;|w|<1\}$, $\T=\{w\in\C;|w|=1\}$, and $\C_+=\{z=x+iy\in\C;y>0\}$. We identify $\C _+/\Z $ with $\D  $ (and $\R/\Z $ with $\T$) using the map $z\mapsto e^{2\pi i z}$. Throughout this paper, we use the generic notation $z=x+iy$ for points in $\C_+$, and $w=e^{2\pi i z}$ for points in $\D$. The respective Lebesgue measures on $\R$ and $\C$ are denoted by $m$ and $A$. It will be convenient to define the weighted area measure $A_1$ on $\D$ by $dA_1(w)=(1-|w|^2)\, dA(w)$.

Given two parametrized sets of nonnegative numbers $\{A_i\}_{i\in I}$ and $\{B_i\}_{i\in I}$, we use the notation $A_i\lesssim B_i$, $i\in I$ to indicate the existence of a positive constant $C$ such that $A_i\le CB_i$ whenever $i\in I$. We then say that $A_i$ is bounded by $B_i$, and refer to $C$ as a bound. Sometimes we allow ourselves to not mention the index set $I$ and instead let it be implicit from the context. If $A_i\lesssim B_i$ and $B_i\lesssim A_i$, then we write $A_i\approx B_i$. We then say that $A_i$ and $B_i$ are comparable.

We let $\DC (\R)$ denote the set $\{[2^{-j}k,2^{-j}(k+1));j,k\in\Z \}$ of dyadic sub intervals of $\R$. The set of dyadic sub intervals of $I\in\DC (\R)$ is denoted by $\DC (I)$. With the identification $[0,1)\simeq\T$ described above, $\DC([0,1))$ is identified with the set $\DC(\T)$ of dyadic sub arcs of $\T$. The Lebesgue measure of $I\in\DC (\R)$ is denoted by $|I|$. The center point, left endpoint and right endpoint of $I\in\DC (\R)$ are denoted by $C_I,L_I$ and $R_I$ respectively. The rank of $I\in\DC (\R)$ is defined as rk$(I)=-\log _2(|I|)$. The $k$th generation of $I\in\DC (\R)$ is defined as $\DC _k(I)=\{J\in\DC (I);|J|=2^{-k}|I|\}$. If $I,J\in\DC (\R)$, and $|I|\le |J|$, then we define the relative distance between $I$ and $J$ as rd$(I,J)=|n|$, where $n$ is the unique number such that $I\subset J+n|J|$. Given $I\in\DC (\T)$, the corresponding Carleson square is given by 
\[
Q_I=\left\{w=e^{2\pi i (x+iy)}\in\D;x\in I, 0\le y\le -\frac{\log (1-|I|)}{2\pi}\right\}.
\]
We also define its half plane correspondent 
\[
\tilde Q_I=\left\{x+iy\in\C_+;x\in I, 0\le y\le -\frac{\log (1-|I|)}{2\pi}\right\}.
\]

The Poisson kernel for $\C _+$ is the function
\begin{equation*}
	P_y^{\C_+}(t)=\frac{1}{\pi}\frac{y}{t^2+y^2},\quad y>0,\ t\in\R.
\end{equation*}
We define the Poisson extension (to $\C _+$) of a suitable function $f:\R\to\C$ as
\[
f(z)=\int_\R f(t) P_y^{\C_+}(x-t)\, dt,\quad z=x+iy\in\C_+.
\]

The Fourier transform of an integrable function $f:\R\to\C$ is given by
\[
\FC f(\xi)=\hat f(\xi)=\int_\R f(x)e^{-2\pi i x\xi}\, dx,\quad \xi\in\R.
\]
We recall that $\FC P_y^{\C_+}(\xi)=e^{-2\pi |\xi|y}$. Let $\SC$ denote the Schwartz class of functions defined on $\R$. For $f\in\SC$, we define the analytic and the anti-analytic projections of $f$ as $f^+=P_+f=\FC^{-1}(\1_{\xi>0}\hat f)$ and $f^-=P_-f=\FC^{-1}(\1_{\xi<0}\hat f)$. As one might guess, the respective Poisson extensions of $f^+$ and $f^-$ are analytic and anti-analytic. We also define the Hilbert transform $Hf=-if^++if^-$.

We define the Wirtinger type differential operators $\partial =\partial_x-i\partial_y$, $\bar \partial =\partial_x+i\partial_y$, and the Laplacian $\Delta=\partial\bar{\partial}$. If $f$ is the Poisson extension of a Schwartz function, then we define $Df=-i\partial f^++i\bar{\partial}f^-=\FC^{-1}(\xi\mapsto 4\pi |\xi| \hat f(\xi))$.

Given a function $f:\R\to\C$, we define the periodization $g:\R\to\C$ by
\[
g(x)=\sum_{k\in\Z}f(x-k),\quad x\in\R.
\]
If $f$ is integrable, then it holds that
\[
\int_0^1g(x)e^{-2\pi ixn}\, dx=\hat f(n),\quad n\in\Z.
\]
This implies in particular that $P_{e^{-2\pi y}}$ is the periodization of $P_y^{\C_+}$. Thus, for the Poisson extension of $g$ it holds that
\begin{equation}\label{Eq:Periodization}
	g(w)=\sum_{n=0}^\infty \hat f(n) w^n+\sum_{n=-\infty}^{-1}\hat f(n)\bar w^n=\sum_{k\in\Z}f(z-k),\quad w=e^{2\pi i z}\in\D.
\end{equation}
We will use this in plenty.

Given $x,y\in\HC$, we define the linear rank one operator $x\otimes y:\HC\ni z\mapsto x\langle z,y\rangle_\HC\in\HC$. Note that the tensor product defined in this way has anti-linear dependence on its second factor.

\section{Hankel-type operators, and $BMOA$}\label{Sec:BMOA}

Let $\phi:\D\to\HC$ be an analytic function, with Taylor series representation $\phi(w)=\sum_{n=0}^\infty \hat \phi(n) w^n$, $w\in\D$. Given $\alpha>0$, we define the corresponding order fractional derivative of $\phi$ by $D^\alpha \phi(w)=\sum_{n=0}^\infty (1+n)^\alpha\hat \phi(n) w^n$, $w\in\D$. Also, we define the Hankel operator $\Gamma_\phi$ by the action
\[
\Gamma_\phi f(w)=\sum_{n=0}^\infty \left(\sum_{m=0}^\infty \hat \phi(m+n)\hat f(n)\right) w^n,\quad w\in\D,
\]
where $f$ is $\C$-valued, and analytic in a neighborhood of $\overline \D$.

Operators of the type $\Gamma_\phi D^\alpha$ appear naturally in control theory, specifically in the study of weighted admissibility, e.g. \citelist{\cite{Jacob-Rydhe-Wynn2014:WeightWeissConjRKTGenHankOps}\cite{Jacob-Partington-Pott-Wynn2015:BAdmGContrSemGrps}}. Hardy space boundedness properties of $\Gamma_\phi D^\alpha$ have been studied in \cite{Rydhe2016:VecHankOpsCarlesonEmbsBMOA} (see also \cite{Janson-Peetre1988:Paracomms} for the case $\HC=\C$). This study lead to different notions of $BMOA$, bounded mean oscillation of analytic functions. Consider the following three conditions:
\begin{itemize}
	\item[$(i)$] There exists $C>0$ such that for all $w_0\in\D$
	\begin{equation}\label{Eq:BMOAC}
	\int_\D\|(D\phi)(w)\|_\HC^2\frac{(1-|w|^2)}{|1-\bar w w_0|^2}\, dA(w)\le \frac{C^2}{1-|w_0|^2}.
	\end{equation}
	If condition $(i)$ is satisfied, then we say that $\phi\in BMOA_\CC(\HC)$. The space $BMOA_\CC(\HC)$ is equipped with the norm $\|\phi\|_{\CC}=\inf\{C;\textnormal{ \eqref{Eq:BMOAC} holds}\}$.
	\item[$(ii)$] There exists $C>0$ such that for all $f\in H^2(\HC)$ it holds that
	\begin{equation}\label{Eq:BMOAC*}
	\int_\D|\langle f(w),(D\phi)(w) \rangle_\HC|^2(1-|w|^2)\, dA(w)\le C^2\|f\|_{H^2(\HC)}^2.
	\end{equation}
	If condition $(ii)$ is satisfied, then we say that $\phi\in BMOA_{\CC^*}(\HC)$. The space $BMOA_{\CC^*}(\HC)$ is equipped with the norm $\|\phi\|_{\CC^*}=\inf\{C;\textnormal{ \eqref{Eq:BMOAC*} holds}\}$.
	\item[$(iii)$] There exists $C>0$ such that for all $x\in \HC$ and $w_0\in\D$ it holds that
	\begin{equation}\label{Eq:BMOAW}
	\int_\D|\langle x,(D\phi)(w) \rangle_\HC|^2\frac{(1-|w|^2)}{|1-\bar w w_0|^2}\, dA(w)\le \frac{C^2\|x\|_\HC^2}{1-|w_0|^2}.
	\end{equation}
	If condition $(iii)$ is satisfied, then we say that $\phi\in BMOA_\WC(\HC)$. We equip the space $BMOA_\WC(\HC)$ with the norm $\|\phi\|_{\WC}=\inf\{C;\textnormal{ \eqref{Eq:BMOAW} holds}\}$.
\end{itemize}
It is well-known, e.g. \cite{Garnett2007:BddAnalFcnsBook}, that
\[
BMOA_\CC(\C)= BMOA_{\CC^*}(\C)=BMOA_\WC(\C),
\]
with equivalent norms, and, moreover, that $\Gamma_\phi D^\alpha$ is bounded on $H^2(\T,\C)$ if and only if $D^\alpha \phi\in BMOA_\CC(\C)$, e.g. \cite{Peller2003:HankOpsBook}.

If $\HC$ is infinite-dimensional, then we obtain instead the following chain of strict inclusions:
\[
BMOA_\CC(\HC)\subsetneq BMOA_{\CC^*}(\HC)\subsetneq	BMOA_\WC(\HC).
\]
The first inclusion was obtained in \cite{Rydhe2016:VecHankOpsCarlesonEmbsBMOA}. We now justify the second inclusion: It holds that
\[
\|\phi\|_{\CC^*}=\|\PC\|_{H^2(\T,\HC)\to L^2(\D,\HC,(D\phi)\otimes (D\phi)\, dA_1)},
\]
and
\[
\|\phi\|_{\WC^*}\approx \|(D\phi) \otimes (D\phi)\, dA_1\|_\IC.
\]
The first is merely an algebraic reformulation, while the second is a typical exercise, cf. \cite[Lemma VI.3.3]{Garnett2007:BddAnalFcnsBook}. Furthermore, condition $(iii)$ just means that for some $C>0$, \eqref{Eq:BMOAC*} is satisfied for the class of functions $\{k_{w_0}x\}_{w_0\in\D,x\in\HC}$, where $k_{w_0}(w)=\frac{1}{1-\overline{w_0}w}$ are the reproducing kernels for $H^2$. We thus obtain that $BMOA_{\CC^*}(\HC)\subseteq BMOA_\WC(\HC)$. Strictness of the inclusion follows from Theorem \ref{Theorem:Main}. Indeed, if the inclusion was not strict, then the open mapping theorem would imply that the identity operator from $BMOA_\WC(\HC)$ into $BMOA_{\CC^*}(\HC)$ is bounded. This would contradict Theorem \ref{Theorem:Main}. As a result we obtain Corollary \ref{Corollary:Existence}.

The above results also have implications on the existence of so called reproducing kernel theses (RKT) for Hankel-like operators; another concept appearing naturally in control theory. We refer to \citelist{\cite{Partington-Weiss2000:AdmRightShift}\cite{Jacob-Rydhe-Wynn2014:WeightWeissConjRKTGenHankOps}} for details, but point out that by results in \cite{Jacob-Rydhe-Wynn2014:WeightWeissConjRKTGenHankOps}, $D^\alpha \Gamma_\phi$ has an RKT while $\Gamma_\phi D^\alpha$ does not. The inclusion $BMOA_{\CC^*}(\HC)\subsetneq	BMOA_\WC(\HC)$ implies, via results in \cite{Rydhe2016:VecHankOpsCarlesonEmbsBMOA}, that the adjoint operator $(\Gamma_\phi D^\alpha)^*$ also does not have an RKT.

\section{Proof of Theorem \ref{Theorem:Main}}\label{Sec:Proof}
The heuristics of the proof is as follows: Let $\delta_{w}$ denote a point mass at $w\in\D  $. The measure constructed in \cite{Nazarov-Treil-Volberg1997:CounterExInfDimCarlesonEmbThm} is of the form $d\mu = \sum_{I\in\DC }\delta_{w_I} \langle \cdot , \omega_I \rangle \varphi_I$, for some points $\{w_I\}_{I\in\DC (\T)}$ and vectors $\{\omega_I\}_{I\in\DC (\T)}$. If we formally define the function $F=\sum_{I\in\DC (\T)}\delta_{w_I}^{1/2} \omega_I$ then $d\mu=F\otimes F\, dA$. The idea behind the construction to follow is to find functions that behave like ``square roots of point masses'' in the sense that they are well localized, and essentially orthogonal. Our examples of such functions are given by smooth wavelets.

We give an outline of the proof: Let $N$ denote the dimension of $\HC$. In Subsection \ref{Subsec:Harmonic}, we construct a measure $d\nu=\varphi \otimes \varphi\, dA_1$, where $\varphi:\D\to \HC$ is harmonic. In Subsection \ref{Subsec:HarmIntensity}, we state three lemmas, and use these to prove that $\|\nu\|_\IC$ is uniformly bounded in $N$. In Subsection \ref{Subsec:Analytic}, we use $\nu$ to construct $\mu$ such that $d\mu=\phi \otimes \phi\, dA_1$, where $\phi:\D\to\HC$ is analytic. It will follow easily that $\|\mu\|_\IC$ is uniformly bounded in $N$. In Subsection \ref{Subsec:Embedding}, we prove that the corresponding embeddings are bounded below by $(\log N)^{1/2}$. In Subsection \ref{Subsec:Lemmata}, we prove the three lemmas used in Subsection \ref{Subsec:HarmIntensity}.

\subsection{The harmonic construction}\label{Subsec:Harmonic}

A Littlewood-Paley wavelet $\{\psi_I\}_{I\in\DC (\R)}$ is an orthonormal basis for $L^2(\R)$ satisfying the dilation translation relation
\begin{equation}
	\psi_I(x)=\frac{1}{|I|^{1/2}}\psi\Big(\frac{x-C_I}{|I|}\Big) \label{Eq:DilationTranslation},
\end{equation}
where $\psi$ is an even Schwartz function such that	$\hat\psi\ge 0$, $\hat \psi$ has support on $[-\frac{4}{3},-\frac{1}{3}]\cup[\frac{1}{3},\frac{4}{3}]$, and $\hat\psi>0$ on $[\frac{3}{8},\frac{5}{4}]$. Such a wavelet is constructed in \cite[Chapter 3]{Meyer1992:WaveletsOpsBook}.

For $I\in\DC (\R)$, we define the functions $f_I=|I|^{1/2}(D\psi_I)$. For $I\in\DC (\T)$, we define the corresponding periodizations $g_I$ by
\begin{equation*}
g_I(e^{2\pi i x})=\sum_{k\in\Z } f_I(x-k),\qquad x\in\R.
\end{equation*}
The family $\{g_I\}_{I\in\DC (\T)}$ is the first of two main ingredients in the construction.

The second ingredient is a family of vectors which is constructed as follows: Let $\{e_l\}_{l=1}^{N}$ be an orthonormal basis for $\HC$, and define the numbers $a_l=\frac{1}{l(\log N)^{1/2}}$, where $1\le l\le N$. For $I\in\DC (\T)$ with $\rk(I)=j\in[1,N]$, we define
\[
\omega_I=\sum_{l=0}^{j-1}a_{j-l}e_le^{2\pi i 2^lC_I}.
\]
For intervals of other ranks we let $\omega_I=0$.	The function that we want is now given by 
\[
\varphi=\sum_{I\in\DC (\T)}g_I\omega_I.
\]
	
\subsection{The Carleson intensity is good} \label{Subsec:HarmIntensity}

To prove that $\|\nu \|_\IC$ is uniformly bounded in $N$, we need three properties of the functions $g_I$. The first one is just a description of how the norms of these functions scale with the size of $I$, and follows more or less by a change of variables:

\begin{lemma}\label{Lemma:Scaling}
	\begin{equation*}
		\int_{\D  }|g_I|^2\, dA_1\lesssim |I|,\quad I\in\DC(\T).
	\end{equation*}
\end{lemma}

Now consider the measure given by $d\nu=\varphi \otimes \varphi\, dA_1$. Note that 
\begin{equation}\label{Eq:TensorSquare}
\varphi \otimes \varphi =\sum_{I,J\in\DC(\T)}g_I\overline{g_J} (\omega_I\otimes \omega_J ).
\end{equation}	
The diagonal terms of this sum can be estimated using the following:	
\begin{lemma}\label{Lemma:DiagonalTerms}
	\begin{equation*}
	\sum_{I\in\DC (K)}|\langle \omega_I,e\rangle_\HC|^2|I|\lesssim |K|\|e\|^2,\quad K\in\DC (\T),\ e\in\HC.
	\end{equation*}
\end{lemma}	
In \cite{Nazarov-Treil-Volberg1997:CounterExInfDimCarlesonEmbThm}, uniform boundedness of the Carleson intensity is essentially a dyadic version of Lemma \ref{Lemma:DiagonalTerms}. Since the functions $\{g_I\}$ do not have disjoint supports, we will also need to estimate the off-diagonal terms in \eqref{Eq:TensorSquare}. This is the main technical complication of this paper:	
\begin{lemma}\label{Lemma:OffDiagonalTerms}
	\begin{equation*}
	\sum_{\substack{I,J\in\DC (\T)\\
			\neg (I=J\in\DC (K))}}|\int_{Q_K} g_I\overline{g_J}\, dA_1|\lesssim |K|,\quad K\in \DC(\T).
	\end{equation*}
\end{lemma}
	
These lemmas yield a short proof that $\|\nu \|_\IC$ is uniformly bounded: Assume that $\|e\|_\HC\le 1$. Then
\begin{align*}
	\int_{Q_K}\langle d\mu\, e,e\rangle_\HC = {} 
	& \underbrace{\sum_{I\in\DC (K)}\int_{Q_K}|g_I|^2|\langle \omega_I,e\rangle_\HC|^2\, dA_1}_{=:I_1} \\
	&+ \underbrace{\sum_{\substack{I,J\in\DC (\T) \\ \neg (I=J\in\DC (K))}}\int_{Q_K}g_I\overline{g_J}\langle \omega_I,e\rangle_\HC\langle e,\omega_J \rangle_\HC\, dA_1}_{=:I_2}.
\end{align*}
By Lemma \ref{Lemma:Scaling} and Lemma \ref{Lemma:DiagonalTerms}
\begin{align*}
	I_1\le \sum_{I\in\DC (K)}\int_{\D  }|g_I|^2|\langle \omega_I,e\rangle_\HC|^2\, dA_1 	\lesssim \sum_{I\in\DC (K)}|\langle \omega_I,e\rangle_\HC|^2|I|\lesssim |K|.
\end{align*}
The vectors $\{\omega_I\}_{I\in\DC(\T)}$ are easily seen to have less that unit norm, so by Lemma \ref{Lemma:OffDiagonalTerms}
\begin{equation*}
	I_2\le \sum_{\substack{I,J\in\DC (\T) \\ \neg (I=J\in\DC (K))}}|\int_{Q_K}g_I\overline{g_J}\, dA_1|\lesssim |K|.
\end{equation*}

\subsection{Making things analytic}\label{Subsec:Analytic}

Once we have the harmonic construction, the analytic ditto is obtained quite easily. The proof that $\|\nu \|_{CM}$ is uniformly bounded relies on orthogonality and localization of the functions $\{f_I\}_{I\in \DC (\R)}$. The localization in turn is obtained by the translation dilation relation \eqref{Eq:DilationTranslation}, combined with the fact that $\hat \psi$ vanishes in a neighborhood of $0$. The Hilbert transform preserves both orthogonality, and Fourier supports. Let $\tilde f_I=Hf_I$, and $\tilde g_I$ be the corresponding periodization. Repeating the proofs from section \ref{Subsec:HarmIntensity}, with $\{\tilde f_I\}_{I\in\DC(\R)}$ in place of $\{f_I\}_{I\in\DC(\R)}$, one sees that the measures
\[
\bigg(\sum_{I\in\DC (\T)}\tilde g_{I}\omega_{I} \bigg)\otimes \bigg(\sum_{I\in\DC (\T)}\tilde g_{I}\omega_{I} \bigg)\, dA_1
\]
have uniformly bounded Carleson intensity. We now define the analytic functions $f_I^+=\frac{1}{2}(f_I+i\tilde f_I)$, the corresponding periodizations $g_I^+$, 
\[
\phi=\sum_{I\in\DC (\T)}g_I^+\omega_I,
\]
and $d\mu=\phi\otimes \phi dA_1$. The functions $f_I^+$ are analytic and well localized, but not orthogonal. However, for an arbitrary unit vector $e\in\HC$ we have that
\begin{align*}
	|\langle \phi, e\rangle_\HC|^2 \lesssim |\langle \sum_{I\in\DC (\T)}g_I\omega_I, e\rangle_\HC |^2+ |\langle \sum_{I\in\DC (\T)}\tilde g_I\omega_I, e\rangle_\HC |^2.
\end{align*}
It immediately follows that $\|\mu\|_{CM}$ is uniformly bounded in $N$.

\begin{remark}
	It may seem odd to the reader that we do not simply let $\{f_I\}_{I\in \DC (\R)}$ be a family of analytic functions to begin with. The reason for this is that no analytic family, satisfying the additional regularity conditions that we need, is an orthonormal wavelet basis for $H^2(\R)$, as was demonstrated by Auscher \cite{Auscher1995:SolOfTwoProblsOnWavelets}.
\end{remark}

\subsection{Breaking the embedding}\label{Subsec:Embedding}
To prove that the embedding is bad, we follow closely \cite{Nazarov-Treil-Volberg1997:CounterExInfDimCarlesonEmbThm}. Consider the function $E(w)=\sum_{l=1}^Nw^{2^l}e_l$, $w\in\D$. Obviously $\|E\|_{H^2(\T,\HC)}^2=N$. Now
\begin{align*}
	\int_{\D}\langle d\mu\, E,E\rangle_\HC 
	&=
	\sum_{I_1,I_2\in\DC (\T)}\int_{\D}g_{I_1}^+\overline{g_{I_2}^+}\langle \omega_{I_1} ,E\rangle_\HC \langle E,\omega_{I_2} \rangle_\HC\, dA_1
	\\
	&=
	\sum_{\substack{0\le l_1 < j_1\le N \\0\le l_2 < j_2\le N \\ I_1\in\DC _{j_1}(\T)\\I_2\in\DC _{j_2}(\T)}}a_{j_1-l_1}a_{j_2-l_2}e^{2\pi i(2^{l_1}C_{I_1}-2^{l_2}C_{I_2})}\int_\D g_{I_1}^+\overline{g_{I_2}^+}\bar w^{2^{l_1}} w^{2^{l_2}}\, dA_1.
\end{align*}
The integrals are easily computed in terms of Taylor coefficients:
\begin{align*}
	\int g_{I_1}^+\overline{g_{I_2}^+}\bar w^{2^{l_1}} w^{2^{l_2}}\, dA_1
	=
	\pi\sum_{m=-2^{l_1}}^\infty\frac{\hat g_{I_1}^+(m+2^{l_1})\overline{\hat g_{I_2}^+(m+2^{l_2})}}{(m+2^{l_1}+2^{l_2}+1)(m+2^{l_1}+2^{l_2}+2)}.
\end{align*}
We consider fixed $j_1,j_2,l_1,l_2$, and use that $\hat g_I(n)=4\pi n|I|\hat\psi(n|I|)e^{-2\pi i nC_I}$ to compute
\begin{equation}\label{Eq:Sum}
	\sum_{\substack{I_1\in\DC _{j_1}(\T)\\I_2\in\DC _{j_2}(\T)}}e^{2\pi i(2^{l_1}C_{I_1}-2^{l_2}C_{I_2})}\int_\D g_{I_1}^+\overline{g_{I_2}^+}\bar w^{2^{l_1}} w^{2^{l_2}}\, dA_1 
	= 
	\sum_{m=-2^{l_1}}^\infty\alpha_m\beta_m.
\end{equation}
where
\[
\alpha_m=\frac{16\pi^32^{-j_1-j_2}(m+2^{l_1})(m+2^{l_2})\hat \psi^+ (\frac{m+2^{l_1}}{2^{j_1}})\overline{\hat \psi^+ (\frac{(m+2^{l_2})}{2^{j_2}})}}{(m+2^{l_1}+2^{l_2}+1)(m+2^{l_1}+2^{l_2}+2)},
\]
and
\[
\beta_m=\sum_{\substack{I_1\in\DC _{j_1}(\T)\\I_2\in\DC _{j_2}(\T)}}e^{-2\pi im(C_{I_1}-C_{I_2})}.
\]
We parametrize $I\in\DC_j(\T)$ by $C_I=(\frac{1}{2}+n)$, $0\le n\le 2^j-1$, and by geometric summation 
\begin{equation}\label{Eq:SumOverGeneration}
\beta_m=
\left\{
\begin{array}{rl}
2^{j_1+j_2}e^{-i\pi\big(\frac{m}{2^{j_1}}-\frac{m}{2^{j_2}}\big)}, &\textnormal{if }m\in 2^{j_1}\Z\cap 2^{j_2}\Z,
\\
0,&\text{otherwise}.
\end{array}
\right.
\end{equation}
This shows that the terms in the right-hand side of \eqref{Eq:Sum} vanish, unless $m=k_12^{j_1}=k_22^{j_2}$ for some $k_1,k_2\in\Z$. Assuming this restriction, we now consider $\alpha_m$. Exploiting the support of $\hat \psi^+$, we see that $\alpha_m$ vanishes, unless $\frac{1}{3}<k_1+2^{l_1-j_1},k_2+2^{l_2-j_2}<\frac{4}{3}$. Since $l< j$, this is only possible if $k_1,k_2\in\{0,1\}$. If $k_1=k_2=0$, then non-vanishing terms are precisely those for which $l_1= j_1-1$ and $l_2= j_2-1$. If $k_1=1$, and $k_2=0$, then $m=2^{j_1}=0$, which is of course impossible. Similarly, if $k_1=0$, and $k_2=1$, then all terms vanish. If $k_1=k_2=1$, then the terms vanish, unless $j_1=j_2=j$ and $l_1,l_2\le j-2$. Tracing back the calculations we have computed that
\begin{align*}
\frac{1}{16\pi^3}&\int_{\D}\langle d\mu\, E,E\rangle_\HC \\ 
= {} &\sum_{\substack{0\le l_1 < j_1 \le N\\ 0\le l_2 < j_2 \le N\\ l_1= j_1-1 \\ l_2= j_2-1}}a_{j_1-l_1}a_{j_2-l_2}
\frac{2^{l_1+l_2}\hat \psi^+ (2^{l_1-j_1})\overline{\hat \psi^+ (2^{l_2-j_2})}}{(2^{l_1}+2^{l_2}+1)(2^{l_1}+2^{l_2}+2)}
\\
&+ 
\sum_{\substack{1\le j\le N\\ 0\le l_1,l_2\le j-2}}a_{j-l_1}a_{j-l_2}
\frac{(2^{j}+2^{l_1})(2^{j}+2^{l_2})\hat \psi^+ (1+2^{l_1-j})\overline{\hat \psi^+ (1+2^{l_2-j})}}{(2^j+2^{l_1}+2^{l_2}+1)(2^j+2^{l_1}+2^{l_2}+2)}
\\
\gtrsim {} &
\sum_{\substack{0\le j\le N\\0\le l_1,l_2\le j-2}}a_{j-l_1}a_{j-l_2}\gtrsim N\log N,
\end{align*}
where the last estimate is an elementary calculation. Assuming the validity of Lemma \ref{Lemma:Scaling} through \ref{Lemma:OffDiagonalTerms}, this completes the proof of Theorem \ref{Theorem:Main}. \qed

\subsection{Properties of $g_I$}\label{Subsec:Lemmata}

Before proving Lemma \ref{Lemma:Scaling} through \ref{Lemma:OffDiagonalTerms}, we establish the following Littlewood-Paley type identity:

\begin{lemma}\label{Lemma:Littlewood-Paley}
	\begin{equation}\label{Eq:Littlewood-Paley}
	\int_{\C _+}f_I(x+iy)\overline{f_J(x+iy)}y\, dxdy=|I|\delta_{IJ}.
	\end{equation}
\end{lemma}
\begin{proof}
	Note that $\Delta |\psi_I|^2=|\partial f_I^+|^2+|\bar \partial f_I^-|^2$. By Cauchy's theorem, $\partial f_I^+$ and $\bar \partial f_I^-$ are orthogonal with respect to the inner product
	\[
	(f,g)\mapsto \int_{\C _+}f(x+iy)\overline{g(x+iy)}y\, dxdy.
	\]
	Applying Green's formula,
	\begin{equation*}
	\int_\Omega (v\Delta u-u\Delta v )\, dV=\int_{\partial \Omega}(v\nabla u-u\nabla v )\, d\vec{S},
	\end{equation*}
	with $\Omega=\C _+$, $u=|\psi_I|^2$ and $v=y$ yields 
	\begin{equation*}
	\int_{\C _+}|D\psi_I(x+iy)|^2y\, dxdy=\int_{\R}|\psi_I(x)|^2\, dx.
	\end{equation*}
	Now \eqref{Eq:Littlewood-Paley} follows by polarization, and orthogonality of the system $\{\psi_I\}_{I\in\DC(\R)}$.
\end{proof}

\subsubsection{Proof of Lemma \ref{Lemma:Scaling}}
This step is completely elementary. We use \eqref{Eq:Periodization}, together with the change of variables $w=e^{2\pi i z}$, and Lemma \ref{Lemma:Littlewood-Paley}:
\begin{align*}
\int_{\D  }|g_I|^2\, dA_1 
&=
4\pi ^2\int_{\tilde Q_{\T}} |\sum_{k\in\Z } f_I(z-k)|^2(1-e^{-4\pi y})e^{-2\pi y}\, dxdy 
\\
&\lesssim
\int_{\tilde Q_{\T}}  |\sum_{k\in\Z } f_I(z-k)|^2y\, dxdy
\\
&=
\sum_{k,l\in\Z }\int_{\tilde Q_{\T}}  f_I(z-k)\overline{f_I(z-l)}y\, dxdy
\\
&=
\sum_{l\in\Z }\int_{\C _+}  f_I(z)\overline{f_I(z-l)}y\, dxdy
\\
&=
\int_{\C _+}  |f_I(z)|^2y\, dxdy
=
|I|.
\end{align*}
\qed

\subsubsection{Proof of Lemma \ref{Lemma:DiagonalTerms}}
	Once again, we follow closely \cite{Nazarov-Treil-Volberg1997:CounterExInfDimCarlesonEmbThm}. Let $e=\sum_{l=1}^N b_le_l$ be a unit vector in $\HC$, and $K\in\DC(\T)$ with $\rk(K)=k$. We begin by choosing $j\ge k$, and sum over $\DC(K)\cap\DC_j(\T)$.
	\[
	\sum_{I\in\DC(K)\cap \DC_j(\T)}|\langle \omega_I,e\rangle|^2=\sum_{l_1,l_2=0}^{j-1}a_{j-l_1}a_{j-l_2}\overline{b_{l_1}}b_{l_2}\sum_{I\in\DC(K)\cap \DC_j(\T)}e^{2\pi i(2^{l_1}-2^{l_2})C_I}
	\]
	If $l_1=l_2$, then
	\[
	\sum_{I\in\DC(K)\cap \DC_j(\T)}e^{2\pi i(2^{l_1}-2^{l_2})C_I}=2^{j-k}=\frac{|K|}{|I|}.
	\]
	If $l_1\ne l_2$, then, like in the calculation of \eqref{Eq:SumOverGeneration}, we obtain
	\[
	\sum_{I\in\DC _j(\T)\cap \DC (K)}e^{2\pi i(2^{l_1}-2^{l_2})C_I}
	=\frac{1-e^{2\pi i (2^{l_1}-2^{l_2})2^{-k}}}{1-e^{2\pi i (2^{l_1}-2^{l_2})2^{-j}}}
	\]
	The above right-hand side will be approximated using the elementary estimate
	\[
	|1-e^{2\pi i x}|\approx |x|,\quad|x|\le \frac{1}{2}.
	\]
	By symmetry, it suffices to consider the case $l_1>l_2$. If $l_1,l_2\ge k$, then $1-e^{2\pi i (2^{l_1}-2^{l_2})2^{-k}}=0$, so any such terms vanish. If $j>l_1\ge k>l_2$, then
	\[
	|\frac{1-e^{2\pi i (2^{l_1}-2^{l_2})2^{-k}}}{1-e^{2\pi i (2^{l_1}-2^{l_2})2^{-j}}}|
	\lesssim
	|\frac{2^{l_2-k}}{(2^{l_1}-2^{l_2})2^{-j}}|\lesssim 2^{l_2-l_1}\frac{|K|}{|I|}.
	\]
	If $l_1,l_2 < k$, then
	\[
	|\frac{1-e^{2\pi i (2^{l_1}-2^{l_2})2^{-k}}}{1-e^{2\pi i (2^{l_1}-2^{l_2})2^{-j}}}|
	\lesssim
	\frac{|K|}{|I|}.
	\]
	With these results
	\[
	\sum_{I\in\DC (K)}|I||\langle \omega_I, e\rangle|^2
	= 
	\sum_{j=k}^N\sum_{I\in\DC _j(K)\cap \DC (\T)}|I||\langle \omega_I, e\rangle|^2 
	\lesssim 
	C|K|,
	\]
	where
	\begin{align*}
	C=\sum_{j=k}^N\left(\sum_{l=0}^{j-1}a_{j-l}^2|b_{l}|^2 +
	\sum_{l_1=k}^{j-1}\sum_{l_2=0}^{k-1}a_{j-l_1}a_{j-l_2}|b_{l_1}||b_{l_2}|2^{l_2-l_1}\right.
	\\
	+
	\sum_{\substack{l_1,l_2=0 \\l_1\ne l_2}}^{k-1}a_{j-l_1}a_{j-l_2}|b_{l_1}||b_{l_2}|\Bigg).
	\end{align*}
	We now make good use of Cauchy--Schwarz's inequality, and rearrangement of terms: First
	\[
	\sum_{j=k}^N\sum_{l=0}^ja_{j-l}^2|b_{l}|^2\le \sum_{l=0}^N\sum_{j=l}^Na_{j-l}^2|b_l|^2 \lesssim  \sum_{l=0}^N|b_l|^2 = 1.
	\]
	Second		
	\begin{align*}
	\sum_{l_1=k}^{j-1}\sum_{l_2=0}^{k-1}a_{j-l_1}a_{j-l_2}|b_{l_1}||b_{l_2}|2^{l_2-l_1}
	&=
	\Bigg(\sum_{l_1=k}^{j-1}a_{j-l_1}|b_{l_1}|2^{-l_1}\Bigg)\Bigg(\sum_{l_2=0}^{k-1}a_{j-l_2}|b_{l_2}|2^{l_2}\Bigg)
	\\
	&\le
	\Bigg(\sum_{l_1=k}^{j-1}a_{j-l_1}^24^{-l_1}\Bigg)^{1/2}
	\Bigg(\sum_{l_2=0}^{k-1}a_{j-l_2}^24^{l_2}\Bigg)^{1/2} \\
	&=
	\Bigg(\sum_{l_1=k}^{j-1}a_{j-l_1}^24^{k-l_1}\Bigg)^{1/2}
	\Bigg(\sum_{l_2=0}^{k-1}a_{j-l_2}^24^{l_2-k}\Bigg)^{1/2}.
	\end{align*}
	Note that $\sum_{l_2=0}^{k-1}a_{j-l_2}^24^{l_2-k}\lesssim \frac{1}{\log N}$, while 
	\[
	\sum_{l_1=k}^{j-1}a_{j-l_1}^24^{k-l_1}
	=
	\sum_{l=0}^{j-k-1}a_{j-k-l}^24^{-l}
	\lesssim
	\sup_{0\le l\le j-k-1}2^{-l}a_{j-k-l}^2
	\lesssim
	\frac{1}{(1+j-k)^2\log N}.		
	\]
	Thus
	\begin{align*}
	\sum_{j=k}^N \sum_{l_1=k}^{j-1}\sum_{l_2=0}^{k-1}a_{j-l_1}a_{j-l_2}|b_{l_1}||b_{l_2}|2^{l_2-l_1}\lesssim \frac{1}{\log N}\sum_{j=k}^N\frac{1}{1+j-k}\lesssim 1.
	\end{align*}
	Third
	\[
	\sum_{j=k}^N	\sum_{\substack{l_1,l_2=0 \\l_1\ne l_2}}^{k-1}a_{j-l_1}a_{j-l_2}|b_{l_1}||b_{l_2}|
	\le
	\sum_{j=k}^N\sum_{l=j-k+1}^ja_l^2 
	\le
	\sum_{l=1}^Nla_l^2
	\lesssim 1.
	\]
This completes the proof of Lemma \ref{Lemma:DiagonalTerms}. \qed
	
\subsubsection{Proof of Lemma \ref{Lemma:OffDiagonalTerms}}

We now address the main technical difficulty of this paper. As a preliminary to Lemma \ref{Lemma:OffDiagonalTerms}, we prove the following result on localization of Poisson extensions for certain Schwartz functions:

\begin{lemma}\label{Lemma:Localization}
	Let $\varphi\in\mathcal{S}$ such that $d=\textnormal{dist}(\textnormal{spt}(\hat\varphi)),0)>0$ and let $p$ be a polynomial of degree $n$. Then
	\begin{equation*}
		|p(x)(\varphi\ast P_y)(x)|\lesssim (1+y^n)\frac{e^{-2\pi dy}}{y^{1/2}},\quad x+iy\in\C _+.
	\end{equation*}
\end{lemma}
\begin{proof}
	By the Fourier inversion formula, and the Leibniz rule,
	\begin{align*}
		p(x)(f\ast P_y)(x) &= \int\bigg(p\Big(\frac{1}{2\pi i}\frac{d}{d\xi}\Big)\big(\hat f(\xi) e^{-2\pi |\xi|y}\big)\bigg)e^{2\pi i x\xi}\, d\xi \\
		&=
		\int \bigg(\sum_{k,l=0}^n a_{kl}\Big(\frac{y\xi}{|\xi|}\Big)^l\hat f^{(k)}(\xi ) e^{-2\pi |\xi|y}\bigg)e^{2\pi ix\xi}\, d\xi,
	\end{align*}
	for some numbers $(a_{kl})_{k,l=0}^n$. Using the decay of $\hat f$ (and its derivatives) along with Cauchy--Schwarz inequality one obtains
	\begin{align*}
		|p(x)(f\ast P_y)(x)| &\lesssim 
		\sum_{k,l=0}^n|a_{kl}||y|^l\int_{d}^\infty \frac{1}{\xi}e^{-2\pi \xi y}\, d\xi \\
		&\lesssim 
		(1+y^n)\bigg(\int_d^\infty e^{-4\pi \xi y}\, d\xi \bigg)^{1/2} =(1+y^n)\frac{e^{-2\pi dy}}{\sqrt{4\pi y}}.
	\end{align*}
\end{proof}	
A few simple manipulations show that
\begin{equation*}
	f_I(x+iy)=\frac{1}{|I|}(D\psi)\ast P_{y/{|I|}}\Big(\frac{x-C_I}{|I|}\Big).
\end{equation*}
Applying Lemma \ref{Lemma:Localization}, with $f=(D\psi)$, $p(x)=1+x^2$, and $d=\frac{1}{3}$, yields
\begin{equation}\label{Eq:Localization}
	|f_I(x+iy)|\lesssim \frac{1}{|I|^{1/2}}\frac{1+\big(\frac{y}{|I|}\big)^2}{1+\big(\frac{x-C_I}{|I|}\big)^2}\frac{e^{-\frac{2\pi y}{3}}}{y^{1/2}}.
\end{equation}	
As in the proof of Lemma \ref{Lemma:Scaling} we obtain that
\begin{align*}
\int_{Q_K}g_I\overline{g_J}\, dA_1 = 	4\pi ^2\int_{\tilde Q_K} \sum_{k,l\in\Z } f_I(z-k)\overline{f_J(z-l)}(1-e^{-4\pi y})e^{-2\pi y}\, dxdy.
\end{align*}
By Taylor's formula, $(1-e^{-4\pi y})e^{-2\pi y}=4\pi y+R(y)$, where $|R(y)|\lesssim y^2$. Applying the triangle inequality a few times we obtain that
\begin{align}
\sum_{\substack{I,J\in\DC (\T) \\ \neg (I=J\in\DC (K))}} |\int_{Q_K}g_I\overline{g_J }\, dA_1| \lesssim {} & 
\sum_{\substack{I,J\in\DC (\R) \\ \neg (I=J\in\DC (K)) \\ |I|,|J|\le 1}}
|\int_{\tilde Q_K}f_I(z)\overline{f_J(z)}y\, dxdy| \label{Eq:MainTerms}
\\
&+ 
\sum_{\substack{I,J\in\DC (\R) \\ \neg (I=J\in\DC (K)) \\ |I|,|J|\le 1}}
\int_{\tilde Q_K}|f_I(z)\overline{f_J(z)}|y^2\, dxdy \label{Eq:RemainderTerms}
\end{align}
The terms in the sum on the right-hand side of \eqref{Eq:MainTerms} will be referred to as the main terms, and the terms in \eqref{Eq:RemainderTerms} as the remainder terms.

We prove that the main terms are controlled by $|K|$. Once this is done the remainder terms are easily handled. By symmetry we may assume that $|I|\le |J|$. We treat a number of different cases, roughly in the order of difficulty.

\paragraph{Case $(i)$:} $|K|<|I|\le|J|\le 1$. If $|K|=1$, then this case is trivial. If not, then $-\log (1-|K|)\lesssim |K|$. Using that the integrand is bounded
\begin{equation}\label{Eq:y-Estimate(i)}
\int_0^{C|K|} \Big(1+\Big(\frac{y}{|I|}\Big)^2\Big)\Big(1+\Big(\frac{y}{|J|}\Big)^2\Big)e^{-\frac{2\pi y}{3}\left(\frac{1}{|I|}+\frac{1}{|J|}\right)}\, dy\lesssim C|K|.
\end{equation}
By the definition of relative distance
\begin{equation}\label{Eq:x-Estimate(i)}
\int_{x\in K} \frac{1}{\big(1+\big(\frac{x-C_I}{|I|}\big)^2\big)\big(1+\big(\frac{x-C_J}{|J|}\big)^2\big)}\, dx\lesssim \frac{|K|}{(1+\rd(I,K)^2)(1+\rd(J,K)^2)}.
\end{equation}
By \eqref{Eq:Localization}, \eqref{Eq:y-Estimate(i)}, and \eqref{Eq:x-Estimate(i)},
\begin{align*}
&|\int_{\tilde Q_K}f_I(z)\overline{f_J(z)}y\,  dA(z)| 
\\
&\lesssim
\frac{1}{|I|^{1/2}|J|^{1/2}} \int_{\tilde Q_K}\frac{\big(1+\big(\frac{y}{|I|}\big)^2\big)\big(1+\big(\frac{y}{|J|}\big)^2\big)e^{-\frac{2\pi y}{3}\left(\frac{1}{|I|}+\frac{1}{|J|}\right)}}{\big(1+\big(\frac{x-C_I}{|I|}\big)^2\big)\big(1+\big(\frac{x-C_J}{|J|}\big)^2\big)}\, dxdy 
\\
&\lesssim
\frac{|K|^2}{|I|^{1/2}|J|^{1/2}} \frac{1}{(1+\rd(I,K)^2)(1+\rd(J,K)^2)}.
\end{align*}
The lengths $|I|$ and $|J|$ are of the form $2^k|K|$ and $2^l|K|$, $k,l\ge 1$. Summing over all lengths and all relative distances  one obtains
\begin{align*}
&\sum_{\substack{ I,J\in\DC (\R) \\ |K|<|I|,|J|\le 1}}
|\int_{\tilde Q_K}f_I(z)\overline{f_J(z)}y\, dA(z)| 
\\
&\lesssim
\sum_{k,l=1}^\infty\sum_{m,n\in\Z }\frac{|K|}{2^{(k+l)/2}} \frac{1}{(1+m^2)(1+n^2)}\lesssim |K|.
\end{align*}

\paragraph{Case $(ii)$:} $|I|\le|K|<|J|\le 1$. By the change of variables $\frac{|I|+|J|}{|I||J|}y\mapsto y$
\begin{equation}\label{Eq:y-Estimate(ii)}
\int_0^\infty \Big(1+\Big(\frac{y}{|I|}\Big)^2\Big)\Big(1+\Big(\frac{y}{|J|}\Big)^2\Big)e^{-\frac{2\pi y}{3}\left(\frac{1}{|I|}+\frac{1}{|J|}\right)}\, dy\lesssim \frac{|I||J|}{|I|+|J|}
\approx |I|.
\end{equation}If $\rd(I,K)\le 1$, then
\[
\int_{x\in K}\frac{1}{1+\big(\frac{x-C_I}{|I|}\big)^2}\, dx\le \int_{\R}\frac{1}{1+\big(\frac{x}{|I|}\big)^2}\, dx=\pi|I|.
\]
If $\rd(I,K)\ge 2$, then
\[
\int_{x\in K}\frac{1}{1+\big(\frac{x-C_I}{|I|}\big)^2}\, dx\le \int_{x\in K}\frac{|I|^2}{|x-C_I|^2}\, dx\le \frac{|I|^2}{|K|(\rd(I,K)-1)^2}.
\]
In either case
\begin{equation}\label{Eq:x-Estimate(ii)}
\int_{x\in K}\frac{1}{1+\big(\frac{x-C_I}{|I|}\big)^2}\, dx\lesssim \frac{|I|}{1+\rd(I,K)^2}.
\end{equation}
By \eqref{Eq:Localization}, \eqref{Eq:y-Estimate(ii)}, the definition of relative distance, and \eqref{Eq:x-Estimate(ii)}
\begin{align*}
|\int_{\tilde Q_K}f_I(z)\overline{f_J(z)}y\, dxdy|
\lesssim
\frac{|I|^{1/2}}{|J|^{1/2}} \frac{1}{1+\rd(J,K)^2}\frac{|I|}{1+\rd(I,K)^2}.
\end{align*}
Now $|J|=2^l|K|$, for $l\ge 1$, while $I\in \DC_k\big(|K|+m|K|\big)$, $k\ge 0$, $m\in\Z$.

\begin{align*}
&\sum_{\substack{ I,J\in\DC (\R) \\ |I|\le|K|<|J|\le 1}}
|\int_{\tilde Q_K}f_I(z)\overline{f_J(z)}y\, dA(z)| 
\\
&\lesssim 
\sum_{k,l=0}^\infty\sum_{\substack{m,n\in\Z \\ I\in\DC _k(K+m|K|)}} \frac{2^{-k/2}2^{-l/2}}{1+n^2}\frac{|I|}{1+\rd(I,K)^2}
\\
&\lesssim
\sum_{l=0}^\infty\sum_{m,n\in\Z }\sum_{k=0}^\infty 2^k\frac{2^{-k/2}2^{-l/2}}{1+n^2}\frac{2^{-k}|K|}{1+m^2} \lesssim |K|.
\end{align*}
	
\paragraph{Case $(iii)$:} $|I|\le |J|\le |K|$, $\rd(J,K)\ge 2$. By \eqref{Eq:Localization}, \eqref{Eq:y-Estimate(ii)}, the definition of relative distance, and \eqref{Eq:x-Estimate(ii)}
\begin{align*}
	|\int_{\tilde Q_K}f_I(z)\overline{f_J(z)}y\, dA(z)| 
	&\lesssim
	\frac{|I|^{1/2}}{|J|^{1/2}}\frac{1}{1+\frac{|K|^2}{|J|^2}\rd(J,K)^2}\frac{|I|}{1+\rd(I,K)^2}
	\\
	&\le
	\frac{|I|^{3/2}|J|^{3/2}}{|K|^2\rd(J,K)^2(1+\rd(I,K)^2)}.
\end{align*}
Now $J\in \DC_l\big(K+n|K|\big)$, $l\ge 0$, $|n|\ge 2$, while $I\in \DC_k\big(K+m|K|\big)$, $k\ge l$, $m\in\Z$.
\begin{align*}
\sum_{\substack{I,J\in\DC (\R) \\ |I|\le |J|\le |K| \\ \rd(J,K)\ge 2}}
|\int_{\tilde Q_K}f_I(z)\overline{f_J(z)}y\, dA(z)| 
&\lesssim
\sum_{k,l=0}^\infty\sum_{\substack{m\in\Z, |n|\ge 2 \\ I\in\DC _k(K+m|K|) \\ J\in\DC _l(K+n|K|)}} \frac{|I|^{3/2}|J|^{3/2}}{|K|^2n^2(1+m^2)} \\
&=
\sum_{k,l=0}^\infty\sum_{\substack{m\in\Z \\ |n|\ge 2}} 2^k2^l \frac{2^{-3k/2}2^{-3l/2}|K|}{n^2(1+m^2)} \lesssim |K|.
\end{align*}

\paragraph{Case $(iv)$:} $|I|\le |J|\le |K|$, $\rd(J,K)\le 1$, $\rd(I,K)\ge 2$. By \eqref{Eq:Localization}, \eqref{Eq:y-Estimate(ii)}, the definition of relative distance, and \eqref{Eq:x-Estimate(ii)}, with $J$ in place of $I$,
\begin{align*}
|\int_{\tilde Q_K}f_I(z)\overline{f_J(z)}y\, dA(z)| 
\lesssim
\frac{|I|^{1/2}|J|^{1/2}}{1+\frac{|K|^2}{|I|^2}\rd(I,K)^2}\le \frac{|I|^{3/2}|J|^{3/2}}{|K|^2\rd(I,K)^2}.
\end{align*}
Summing over $J\in \DC_l\big(|K|+n|K|\big)$, $l\ge 0$, $|n|\le 1$, and $I\in \DC_k\big(|K|+m|K|\big)$, $k\ge l$, $|m|\ge 2$, 
\begin{align*}
\sum_{\substack{I,J\in\DC (\R) \\ |I|\le |J|\le |K| \\ \rd(J,K)\le 1 \\ \rd(I,K)\ge 2}}
|\int_{\tilde Q_K}f_I(z)\overline{f_J(z)}y\, dA(z)| 
&\lesssim
\sum_{k,l=0}^\infty\sum_{\substack{|m|\ge 2, |n|\le 1 \\ I\in\DC _k(K+m|K|) \\ J\in\DC _l(K+n|K|)}} \frac{|I|^{3/2}|J|^{3/2}}{|K|^2\rd(I,K)^2}
\\
&=
\sum_{k,l=0}^\infty \sum_{\substack{|m|\ge 2 \\ |n|\le 1}} 2^k2^l \frac{2^{-3k/2}2^{-3l/2}|K|} {m^2}\lesssim |K|.
\end{align*}
		
\paragraph{Case $(v)$:} $I\in \DC\big(K+m|K|\big)$, $|m|\le 1$, $J\in \DC\big(K+n|K|\big)$, $|n|=1$. By symmetry, it suffices to handle the case $n=1$. We split this case into subcases.

\subparagraph{Subcase $(v1)$:} $I=J\in \DC\big(K+|K|\big)$. By \eqref{Eq:Localization} and \eqref{Eq:y-Estimate(ii)}
\begin{align*}
|\int_{\tilde Q_K}f_I(z)\overline{f_J(z)}y\, dA(z)|
\lesssim
\int_{-\infty}^{R_K}\frac{1}{\big(1+\big(\frac{x-C_I}{|I|}\big)^2\big)^2}\, dx.
\end{align*}
By computation
\begin{align*}
\int_{-\infty}^{R_K}\frac{1}{\big(1+\big(\frac{x-C_I}{|I|}\big)^2\big)^2}\, dx
&= 
\frac{|I|}{8}\bigg(\arctan \Big(\frac{|I|}{C_I-R_K}\Big)-\frac{|I|(C_I-R_K)}{|I|^2+(C_I-R_K)^2}\bigg)
\\
&\lesssim
\frac{|I|^2}{C_I-R_K}.
\end{align*}
We parametrize $I\in\DC _k\big(K+|K|\big)$ by $C_I=R_K+(m+\frac{1}{2})|I|$, where $0\le m\le 2^k-1$.
\begin{align*}
\sum_{I\in\DC (K+|K|)}\int_{\tilde Q_K}|f_I(z)|^2y\, dA(z) 
&\lesssim 
\sum_{k=0}^\infty \sum_{I\in\DC _k(K+|K|)} \frac{|I|^2}{C_I-R_K} 
\\
&\lesssim 
\sum_{k=0}^\infty \sum_{m=0}^{2^k-1} 2^{-k}|K|\frac{1}{m+\frac{1}{2}} 
\\
&\lesssim 
\sum_{k=0}^\infty \log(2^k)2^{-k}|K| \lesssim |K|.
\end{align*}

\subparagraph{Subcase $(v2)$:} $I\in \DC(J)$, $I\ne J$. By \eqref{Eq:Localization} and \eqref{Eq:y-Estimate(ii)}
\begin{equation}\label{Eq:HybridEstimate}
|\int_{\tilde Q_K}f_I(z)\overline{f_J(z)}y\, dA(z)|
\lesssim
\frac{|I|^{1/2}}{|J|^{1/2}}\int_{-\infty}^{R_K}\frac{1}{\big(1+\big(\frac{x-C_I}{|I|}\big)^2\big)\big(1+\big(\frac{x-C_J}{|J|}\big)^2\big)}\, dx.
\end{equation}
By computation
\begin{equation}\label{Eq:x-Estimate(v2)}
\int_{-\infty}^{R_K}\frac{1}{\big(1+\big(\frac{x-C_I}{|I|}\big)^2\big)\big(1+\big(\frac{x-C_J}{|J|}\big)^2\big)}\, dx=\sum_{k=1}^3s_k,
\end{equation}
where
\begin{align*}
s_1
&=
\frac{|I||J|(|J|^2-|I|^2)\Big(|J|\arctan \big(\frac{|I|}{C_I-R_K}\big)-|I|\arctan \big(\frac{|J|}{C_J-R_K}\big)\Big)}{2(|I|^2+|J|^2)(C_I-C_J)^2+(C_I-C_J)^4+(|I|^2-|J|^2)^2},
\\
s_2
&=
\frac{|I||J|(C_I-C_J)^2\Big(|J|\arctan \big(\frac{|I|}{C_I-R_K}\big)+|I|\arctan \big(\frac{|J|}{C_J-R_K}\big)\Big)}{2(|I|^2+|J|^2)(C_I-C_J)^2+(C_I-C_J)^4+(|I|^2-|J|^2)^2},
\\
s_3
&=
\frac{(C_I-C_J)|I|^2|J|^2\log\big(\frac{|J|^2+(C_J-R_K)^2}{|I|^2+(C_I-R_K)^2}\big)}{2(|I|^2+|J|^2)(C_I-C_J)^2+(C_I-C_J)^4+(|I|^2-|J|^2)^2}.
\end{align*}
Approximate the denominators of $s_1$ through $s_3$ by $|J|^4$, and $C_I-C_J$ appearing in the numerators by $|J|$. Then
\begin{align*}
|\int_{\tilde Q_K}f_I(z)\overline{f_J(z)}y\, dA(z)| 
\lesssim {} &
\frac{|I|^{5/2}}{|J|^{1/2}}\frac{1}{C_I-R_K}
\\
&+
\frac{|I|^{5/2}}{|J|^{1/2}}\frac{1}{C_J-R_K}
\\
& +
\frac{|I|^{5/2}}{|J|^{3/2}}|\log\Big(\frac{|J|^2+(C_J-R_K)^2}{|I|^2+(C_I-R_K)^2}\Big)|.
\end{align*}

We begin by summing the terms $\frac{|I|^{5/2}}{|J|^{1/2}}\frac{1}{C_I-R_K}$. If $J$ is adjacent to $K$, then we parametrize $I\in\DC _k(J)$ by $C_I=R_K+(m+\frac{1}{2})|I|$, where $0\le m\le 2^k-1$. 
\begin{align*}
\sum_{\substack{I\in\DC(J)\\ I\ne J}}\frac{|I|^{5/2}}{|J|^{1/2}}\frac{1}{C_I-R_K}
=
\sum_{k=1}^\infty \sum_{I\in\DC_k(J)}\frac{2^{-5k/2}|J|^2}{C_I-R_K}
\\
=
|J|\sum_{k=1}^\infty 2^{-3k/2}\sum_{m=0}^{2^k-1}\frac{1}{m+\frac{1}{2}}\lesssim |J|\sum_{k=1}^\infty k2^{-3k/2}\lesssim |J|.
\end{align*}

Clearly the sum of $|J|$ for $J$ adjacent to $K$ is controlled by $|K|$. If $J$ is not adjacent to $K$, then $\frac{1}{C_I-R_K}\approx \frac{1}{C_J-R_K}$. We parametrize $J\in \DC_l\big(K+|K|\big)$, with $\dist(K,J)>0$, by $C_J=R_K+(n+\frac{1}{2})|J|$, where $0\le n\le 2^l-1$.
\begin{align*}
\sum_{\substack{J\in\DC(K+|K|)\\ \dist(K,J)>0\\ I\in\DC(J), I\ne J}}\frac{|I|^{5/2}}{|J|^{1/2}}\frac{1}{C_J-R_K}
=
|K|\sum_{l=0}^\infty \sum_{k=1}^\infty \sum_{n=0}^{2^l-1}2^{-5k/2}2^{-l}\frac{1}{n+\frac{1}{2}}\lesssim |K|.
\end{align*}

Summing the terms $\frac{|I|^{5/2}}{|J|^{1/2}}\frac{1}{C_I-R_K}$ is similar. In order to control the logarithmic terms, note that
\begin{align*}
|\log\Big(\frac{|J|^2+(C_J-R_K)^2}{|I|^2+(C_I-R_K)^2}\Big)| \lesssim
\left\{
\begin{array}{rl}
\log\big(\frac{|J|}{|I|}\big) & \text{if $J$ is adjacent to $K$,}\\
\frac{|J|^2}{(C_J-R_K)^2} & \text{if $J$ is not adjacent to $K$.}
\end{array} \right.
\end{align*}
The terms may now be summed as before.

\subparagraph{Subcase $(v3)$:} $|I|\le|J|$, $I\in\DC\big(K+|K|\big)\setminus\DC(J)$. Again we use \eqref{Eq:HybridEstimate} and \eqref{Eq:x-Estimate(v2)}, but we approximate the denominators of $s_1$ through $s_3$ by $(C_I-C_J)^4$ instead of $|J|^4$:
\begin{align}
|\int_{\tilde Q_K}f_I(z)\overline{f_J(z)}y\, dA(z)|
\lesssim {} &
\frac{|I|^{5/2}|J|^{7/2}}{(C_I-C_J)^4}\frac{1}{C_I-R_K}\label{iii'1}
\\
&+
\frac{|I|^{5/2}|J|^{7/2}}{(C_I-C_J)^4}\frac{1}{C_J-R_K}\label{iii'2}
\\
&+
\frac{|I|^{5/2}|J|^{3/2}}{(C_I-C_J)^2}\frac{1}{C_I-R_K}\label{iii'3}
\\
&+
\frac{|I|^{5/2}|J|^{3/2}}{(C_I-C_J)^2}\frac{1}{C_J-R_K}\label{iii'4}
\\
&+
\frac{|I|^{5/2}|J|^{3/2}}{|C_I-C_J|^3}|\log \Big(\frac{|J|^2+(C_J-R_K)^2}{|I|^2+(C_I-R_K)^2}\Big)|\label{iii'5}.
\end{align}

We being with summing the right-hand side of \eqref{iii'1}. Let $J_0=[R_K,R_K+|J|)$. We use the parametrization $J\in\DC_l\big(K+|K|\big)$, $C_J=R_K+(n+\frac{1}{2})|J|$, $0\le n\le 2^l-1$, $I\in\DC_k\big(J_0+m|J|\big)$, $k\ge 0$, $0\le m\le 2^l-1$, $J_0+m|J|\ne J$. 

If $m\ge 1$, then
\[
\frac{|I|^{5/2}|J|^{7/2}}{(C_I-C_J)^4}\frac{1}{C_I-R_K}
\approx\frac{2^{-5k/2}2^{-l}|K|}{m(n-m)^4}. 
\]
If $m=1$, then we use the additional parametrization $C_I=R_K+(p+\frac{1}{2})|I|$, $0\le p\le 2^k-1$, and
\[
\frac{|I|^{5/2}|J|^{7/2}}{(C_I-C_J)^4}\frac{1}{C_I-R_K}
\approx
\frac{2^{-3k/2}2^{-2l}|K|}{pn^4}. 
\]
Summing,
\begin{align*}
&\sum_{\substack{J\in\DC(K+|K|) \\ |I|\le |J| \\I\in\DC(K+|K|)\setminus\DC(J)}}\frac{|I|^{5/2}|J|^{7/2}}{(C_I-C_J)^4}\frac{1}{C_I-R_K}
\\
&\lesssim 
\sum_{k,l=0}^\infty\bigg( \sum_{p=0}^{2^k-1}\sum_{n=1}^{2^l-1}\frac{2^{-3k/2}2^{-2l}|K|}{pn^4} +
\sum_{m=1}^{2^l-1}\sum_{\substack{n=0\\ n\ne m}}^{2^l}\frac{2^{-3k/2}2^{-l}|K|}{m(n-m)^4}\bigg)
\lesssim
|K|.
\end{align*}

Summing the terms in \eqref{iii'2}, \eqref{iii'3} and \eqref{iii'4}, is similar. For the logarithmic terms \eqref{iii'5}, we have that
\[
\frac{|I|^{5/2}|J|^{3/2}}{|C_I-C_J|^3}|\log \Big(\frac{|J|^2+(C_J-R_K)^2}{|I|^2+(C_I-R_K)^2}\Big)|
\approx
\frac{|I|^{5/2}}{|J|^{3/2}}\frac{|\log\big(\frac{|J|^2+(n+\frac{1}{2})^2|J|^2}{|I|^2+(C_I-R_K)^2}\big)|}{|m-n|^3}
\]
If $m=0$, then
\[
|\log\Big(\frac{|J|^2+(n+\frac{1}{2})^2|J|^2}{|I|^2+(C_I-R_K)^2}\Big)|\le \log \Big(\big(1+(2^l-\frac{1}{2})^2\big)\frac{|J|^2}{|I|^2}\Big)\lesssim (1+k)(1+l).
\]
Similarly, if $n=0$, then
\[
|\log\Big(\frac{|J|^2+(n+\frac{1}{2})^2|J|^2}{|I|^2+(C_I-R_K)^2}\Big)|
=
| \log \Big(\frac{\frac{5}{4}|J|^2}{|I|^2+(C_I-R_K)^2}\Big)|\lesssim (1+k)(1+l).
\]

If $m,n\ge 1$, then
\[
1\le \frac{|J|^2+(n+\frac{1}{2})^2|J|^2}{|I|^2+(C_I-R_K)^2}\le \Big(\frac{n+1}{m}\Big)^2,
\]
whenever $m\le n-1$, and 
\[
\Big(\frac{n}{m+1}\Big)^2\le \frac{|J|^2+(n+\frac{1}{2})^2|J|^2}{|I|^2+(C_I-R_K)^2}\le 1,
\]
whenever $m\ge n+1$. It follows that 	
\[
|\log\Big(\frac{|J|^2+(n+\frac{1}{2})^2|J|^2}{|I|^2+(C_I-R_K)^2}\Big)| \lesssim \frac{|n-m|}{\min\{m,n\}}.
\]

We now compute the sum
\begin{align*}
&\sum_{\substack{J\in\DC(K+|K|) \\ |I|\le |J| \\I\in\DC(K+|K|)\setminus\DC(J)}}\frac{|I|^{5/2}|J|^{3/2}}{|C_I-C_J|^3}|\log \Big(\frac{|J|^2+(C_J-R_K)^2}{|I|^2+(C_I-R_K)^2}\Big)|
\\
&\lesssim
|K|\sum_{k,l=0}^\infty 2^{-3k/2-l}\bigg(\sum_{m=1}^{2^l-1}\frac{(1+k)(1+l)}{m^3}+\sum_{\substack{m,n=1\\m\ne n}}^{2^{l}-1}\frac{1}{|n-m|^2\min\{m,n\}}\bigg)
\\
&\lesssim
|K|\sum_{k,l=0}^\infty 2^{-3k/2-l}(1+k)(1+l)\lesssim |K|.
\end{align*}
	
\subparagraph{Subcase $(v4)$:} $J\in\DC\big(K+|K|\big)$, $I\in\DC\big(K+m|K|\big)$, $m\in\{-1,0\}$. This case is similar to $(v3)$, but when $C_I<R_K$, we need to replace \eqref{Eq:x-Estimate(v2)} with	
\[
\int_{-\infty}^{R_K}\frac{1}{\big(1+\big(\frac{x-C_I}{|I|}\big)^2\big)\big(1+\big(\frac{x-C_J}{|J|}\big)^2\big)}\, dx=\sum_{k=1}^4s_k,
\]
with $s_1,s_2,s_3$ as before, and
\[
s_4
=
\pi \frac{|I||J|^2(|J|^2-|I|^2+(C_I-C_J)^2)}{2(|I|^2+|J|^2)(C_I-C_J)^2+(C_I-C_J)^4+(|I|^2-|J|^2)^2},
\]
The terms $s_1,s_2,s_3$ are summed as in subcase $(v3)$. To sum the terms $s_4$, we parametrize $J\in\DC_k\big(K+|K|\big)$, $k\ge 0$, by $C_J=R_K+(n+\frac{1}{2})|J|$, $0\le n\le 2^{l}-1$, and let $I\in \DC_l\big(J_0-p|J|\big)$, $l\ge 0$, $1\le p\le 2^{l+1}$.
\begin{align*}
\sum_{J\in\DC(K+|K|)}\sum_{m=-1}^0\sum_{\substack{I\in\DC(K+m|K|)\\ |I|\le|J|}}s_4
&\lesssim
\sum_{k,l=0}^\infty \sum_{p=1}^{2^{l+1}}\sum_{\substack{J\in\DC_l(K+|K|)\\ I\in \DC_k(J_0-p|J|)}}\frac{|I|^{3/2}|J|^{3/2}}{(C_I-C_J)^2}
\\
&\lesssim
\sum_{k,l=0}^\infty \sum_{p=1}^{2^{l+1}} \sum_{n=0}^{2^l-1}2^k\frac{2^{-3k/2}2^{-3l/2}|K|}{(n+p)^2}\lesssim |K|.
\end{align*}

\paragraph{Case $(vi)$:} $J\in\DC(K)$, $I\in\DC\big(K+m|K|\big)$, $|m|=1$. This case is similar to case $(v)$.

\paragraph{Case $(vii)$:} $I,J\in\DC(K)$, $I\ne J$. By Lemma \ref{Lemma:Littlewood-Paley} we have that
\begin{align}
&\sum_{\substack{I,J\in\DC (K) \\ I\ne J}}|\int_{\tilde Q_K}f_I(z)\overline{f_J(z)}y\, dA(z)| \nonumber
\\
= {} &
\sum_{\substack{I,J\in\DC (K) \\ I\ne J}}|\int_{\tilde Q_K^c}f_I(z)\overline{f_J(z)}y \, dA(z)| \nonumber
\\
\le {} &
\sum_{\substack{I,J\in\DC (K) \\ I\ne J}}|\int_{x\notin K}\int_{y=0}^{\infty}f_I(z)\overline{f_J(z)}y\,  dA(z)| \label{v1}
\\
&+
\sum_{\substack{I,J\in\DC (K) \\ I\ne J}}|\int_{x\in K}\int_{y=|K|}^{\infty}f_I(z)\overline{f_J(z)}y\,  dA(z)|.\label{v2}
\end{align}
The sum \eqref{v1} is rewritten as
\begin{align*}
&\sum_{\substack{I,J\in\DC (K) \\ I\ne J}}|\int_{x\notin K}\int_{y=0}^{\infty}f_I(z)\overline{f_J(z)}y\, dA(z)|
\\
&\le
\sum_{n\ne 0}\sum_{\substack{I,J\in\DC (K) \\ I\ne J}}|\int_{x\in K+n|K|}\int_{y=0}^{\infty}f_I(z)\overline{f_J(z)}y\, dA(z)|
\\
&=
\sum_{n\ne 0}\sum_{\substack{I,J\in\DC (K+n|K|) \\ I\ne J}}|\int_{x\in K}\int_{y=0}^{\infty}f_I(z)\overline{f_J(z)}y\, dA(z)|
\\
&\le
\sum_{n\ne 0}\sum_{m\in\Z }\sum_{\substack{J\in\DC (K+n|K|) \\ I\in\DC (K+m|K|)}}\int_{x\in K}\int_{y=0}^{\infty}|f_I(z)\overline{f_J(z)}|y\, dA(z) \lesssim |K|,
\end{align*}
as follows from the cases $(i)-(vi)$. 

The terms in \eqref{v2} are approximated by
\begin{align*}
&\int_{x\in \R}\int_{y=|K|}^{\infty}\frac{\big(1+\big(\frac{y}{|I|}\big)^2\big)\big(1+\big(\frac{y}{|I|}\big)^2\big)e^{-\frac{2\pi y}{3}\left(\frac{1}{|I|}+\frac{1}{|J|}\right)}}{\big(1+\big(\frac{x-C_I}{|I|}\big)^2\big)\big(1+\big(\frac{x-C_I}{|I|}\big)^2\big)}\, dydx 
\\
&\lesssim
\frac{|I|^{3/2}}{|J|^{1/2}}e^{-\frac{\pi |K|}{3}\left(\frac{1}{|I|}+\frac{1}{|J|}\right)}\int_{y=0}^\infty \Big(1+\Big(\frac{y}{|I|}\Big)^2\Big)\Big(1+\Big(\frac{y}{|I|}\Big)^2\Big)e^{-\frac{\pi y}{3}\left(\frac{1}{|I|}+\frac{1}{|J|}\right)}\, dy 
\\
&\lesssim
\frac{|I|^{3/2}}{|J|^{1/2}}e^{-\frac{\pi |K|}{3}\left(\frac{1}{|I|}+\frac{1}{|J|}\right)}.
\end{align*}
Summing one now gets that
\begin{align*}
\sum_{\substack{I,J\in\DC (K) \\ I\ne J}}&\int_{x\in K}\int_{y=|K|}^{\infty}|f_I(z)\overline{f_J(z)}|y\, dA(z)
\\
&\lesssim
\sum_{k,l=0}^\infty 2^k2^l2^{-3k/2}2^{-l}e^{-\frac{\pi}{3}\left(2^k+2^l\right)}|K|\lesssim |K|.
\end{align*}

We have proved that the right-hand side of \eqref{Eq:MainTerms} is controlled by $|K|$. This leaves us with \eqref{Eq:RemainderTerms}. We've already done most of the computational work, except that we gain an additional factor $y$. In case $(i)$ we get instead
\begin{align*}
\int_{y=0}^{2|K|}\Big(1+\Big(\frac{y}{|I|}\Big)^2\Big)\Big(1+\Big(\frac{y}{|J|}\Big)^2\Big)e^{-\frac{2\pi y}{3}\left(\frac{1}{|I|}+\frac{1}{|J|}\right)}y\, dy \lesssim |K|^2.
\end{align*}
This is an additional factor $|K|$ which does not give a worse estimate. 
In cases $(ii)-(vi)$ we instead gain an extra factor $\frac{|I||J|}{|I|+|J|}\approx |I|\le |K|$, so these cases are also fine. Only in case $(vii)$ we need to do a little bit of work:
\begin{align*}
\sum_{\substack{I,J\in\DC (K) \\ I\ne J}} \int_{\tilde Q_K}|f_I(z)\overline{f_J(z)}|y^2\, dA(z)
&\lesssim 
\sum_{\substack{I,J\in\DC (K)}} \frac{|I|^{5/2}}{|J|^{1/2}}\frac{1}{1+\rd(I,J)^2} 
\\
&\le
\sum_{k,l=0}^\infty \sum_{n\in\Z }2^k2^l2^{-5k/2}2^{-2l}\frac{|K|^2}{1+n^2}
\lesssim
|K|.
\end{align*}
This completes the proof of Lemma \ref{Lemma:OffDiagonalTerms}, and thus of Theorem \ref{Theorem:Main}. \qed

\bibliographystyle{plain}
%\begin{bibsection}
%	\begin{biblist}
%		\bibselect{paper04bib}
%	\end{biblist}
%\end{bibsection}

\bibliography{paper04bib}

\end{document}